\documentclass[12pt,a4paper,reqno]{amsart}

\usepackage{amsmath}
\usepackage{amssymb}
\usepackage{amsthm}
\usepackage{mathrsfs}
\usepackage{enumerate}
\usepackage[backref=true,dvipdfmx]{hyperref}

\DeclareMathOperator*{\esssup}{esssup}
\DeclareMathOperator*{\essinf}{essinf}

\newtheorem{thm}{Theorem}[section]
\newtheorem{lem}[thm]{Lemma}
\newtheorem{cor}[thm]{Corollary}

\newtheorem{athm}{Theorem}

\theoremstyle{definition} \newtheorem{rk}[thm]{Remark}
\theoremstyle{definition} \newtheorem{df}[thm]{Definition}

\allowdisplaybreaks
\begin{document}

\title{Matrix weighted Kolmogorov-Riesz's compactness theorem}

\author{Shenyu Liu, Dongyong Yang and  Ciqiang Zhuo}

\address{Shenyu Liu, School of Mathematical Sciences\\
 Xiamen University\\
 Xiamen 361005, China}

\email{shenyuliu@stu.xmu.edu.cn}

\address{Dongyong Yang, School of Mathematical Sciences\\
 Xiamen University\\
 Xiamen 361005, China}

\email{dyyang@xmu.edu.cn}

\address{Ciqiang Zhuo(Corresponding author), School of Mathematics and Statistics\\
 Hunan Normal University\\
 Changsha, Hunan 410081, China}

\email{cqzhuo87@hunnu.edu.cn}

\makeatletter
\@namedef{subjclassname@2020}{\textup{2020} Mathematics Subject Classification}
\makeatother

\date{\today}


\begin{abstract}
In this paper, several versions of the Kolmogorov-Riesz compactness theorem in weighted Lebesgue spaces with matrix weights are obtained. In particular, when the matrix weight $W$ is in the known $A_p$ class, a characterization of totally bounded subsets in $L^p(W)$ with $p\in(1, \infty)$ is established.
\end{abstract}

\maketitle

\section{Introduction}
In this paper, we investigate totally bounded sets in matrix weighted Lebesgue spaces, from which one can obtain corresponding compactness criteria via the Hausdorff criterion for compactness, that is, a set is precompact if and only if it is complete and totally bounded.\par

In classical Lebesgue spaces $L^p$, the characterization of precompact sets was given by the celebrated Kolmogorov-Riesz theorem (see \cite{HH}) which was first discovered by Kolmogorov \cite{K} in $L^p([0,1])$ for $p\in (1,\infty)$. Subsequently, Tamarkin \cite{Ta} extended the result to the case in which the underlying space can be unbounded, with an additional condition related to the behaviour at infinity. Tulajkov \cite{Tu} showed that Tamarkin's result was also true when $p=1$. At the same time, Riesz \cite{Ri} independently proved a similar result. Since then, compactness criteria of subsets in Lebesgue spaces have been studied and applied in various settings, e.g. see \cite{S,WY} for some improvements and applications on Kolmogorov-Riesz's theorem, and see \cite{Ra,GM2,GR,Ba,BG,AU} for a series of works on compactness criteria in variable exponent function spaces.

In particular, Tsuji \cite{T} showed that Kolmogorov-Riesz's theorem is true in $L^p(\mathbb{R})$ for $p\in (0,1)$ and his method has been applied by many authors, e.g. \cite{XYY,GZ,COY}. Moreover, based on the Arzel\'{a}-Ascoli theorem (see \cite{GN}), the authors in \cite{Si,GN} established compactness criteria in Lebesgue-Bochner spaces. In addition, the authors in \cite{Kat,Kr} applied the so-called Lebesgue-Vitali's theorem and established compactness criteria in $L^0(m)$, the space of all Lebesgue measurable functions on $\mathbb{R}^n$ that are finite almost everywhere, where $m$ denotes the Lebesgue measure.\par

Recently, Hanche-Olsen--Holden \cite{HH} seminally showed that both the Arzel\'{a}-Ascoli theorem and Kolmogorov-Riesz's theorem are consequences of a simple lemma on compactness in metric spaces via a finite dimension argument. Inspired by the method used in \cite{HH}, Clop--Cruz \cite{CC} first gave a compactness criterion in scalar weighted Lebesgue spaces $L^p(\omega)$ for $p\in (1,\infty)$ with a weight $\omega\in A_p$. Their result was then improved by Guo-Zhao in \cite{GZ}, in which they gave the following compactness criterion in $L^p(\omega)$ for $p\in(0,\infty)$ with $\omega\in L^1_{\rm{loc}}(\mathbb{R}^n)$.
\begin{athm}\label{thm-GZ}
Let $0<p<\infty$ and $\omega\in L^1_{\rm{loc}}(\mathbb{R}^n)$ be a nonnegative function. A subset $\mathcal{F}$ of $L^p(\omega)$ is totally bounded if the following conditions hold:
\begin{enumerate}
\item[(a)] $\mathcal{F}$ is bounded, i.e. $\sup\limits_{f\in \mathcal{F}}\|f\|_{L^p(\omega)}<\infty;$
\item[(b)] $\mathcal{F}$ uniformly vanishes at infinity, that is, $$\lim_{R\to \infty}\sup_{f\in \mathcal{F}}\|f\chi_{B^c(0,R)}\|_{L^p(\omega)}=0;$$
\item[(c)] $\mathcal{F}$ is equicontinuous, that is, $$\lim_{r\to 0}\sup_{f\in \mathcal{F}}\sup_{y\in B(0,r)}\|\tau_yf-f\|_{L^p(\omega)}=0.$$
\end{enumerate}
Here, $\tau_y$ denotes the translation operator: $\tau_yf(x):=f(x-y)$.
\end{athm}\par

There is a natural question whether Theorem \ref{thm-GZ} can be extended to the setting of matrix weights. As a natural vector-valued generalization of scalar Muckenhoupt $A_p$ weights, the theory of matrix weights was first introduced by Bloom \cite{B1,B2} in 1981. Then the theory was pushed forward through the seminal work of Nazarov-Treil-Volberg \cite{NT,V,TV}, Christ-Goldberg \cite{CG,G}, and Frazier-Roudenko \cite{R,FR} in the late 1990s, that was arose from problems in the theory of stationary processes, the theory of Toeplitz operators and multivariable elliptic PDEs. From then on, harmonic analysis with matrix weights have been considered by many authors in various directions. For the references, we refer to \cite{CIM,CIMPR} for recent developments on matrix weights, \cite{NPTV,HPV,I} for the matrix $A_2$ conjecture related to the sharp norm estimates for singular operators, and \cite{CMR,DLY} for some applications of matrix weights.\par

Matrix weights share many properties with scalar weights, for example, the definition of matrix $A_p$ weights due to Frazier-Roudenko \cite{R,FR} seems to be an intuitive extension of the definition of scalar $A_p$ weights and the Hilbert transform is bounded on $L^p(W)$ if and only if $W\in A_p$. However, due to the non-commutativity in the matricial setting, many techniques of the classical harmonic analysis fail to generalize to the case of vector-valued functions with matrix weights, so the vector-valued case cannot be easily reduced to the scalar case. For example, in the setting of matrix weights, a suitable theory of weak-type spaces $L^{p,\infty}(W)$ is unknown.\par

In this paper, we obtain several generalizations of Theorem \ref{thm-GZ} in weighted Lebesgue spaces with matrix weights. To be precise, we first obtain a Kolmogorov-Riesz theorem in the weighted variable Lebesgue space $L^{p(\cdot)}(\rho)$ on $(\mathbb R^n, |\cdot|, m)$, where $\rho=\{\rho_x\}_{x\in\mathbb R^n}$ is a family of norms on $\mathbb C^d$, and another version of compactness criterion in $L^p(W)$ when $p\in (0,\infty)$ and $W$ is a matrix weight on $\mathbb R^n$. We also establish a  Kolmogorov-Riesz theorem in $L^p(\rho, \mu)$ with $p\in[1, \infty)$ on metric measure spaces $(X, d, \mu)$, and obtain an equivalent characterization of precompact subsets in $L^p(W)$ for $p\in (1, \infty)$ and $W$ in $A_p$ class on $\mathbb R^n$ as an application.\par

We would like to emphasize that due to the special structure of matrix weighted Lebesgue spaces, neither matrix weighted Lebesgue spaces $L^p(W)$ nor weighted Lebesgue spaces $L^p(\rho)$ are in the framework of (quasi-)Banach function spaces in \cite{GR,CGO,GZ} or Lebesgue-Bochner spaces in \cite{Si,GN}. For example, in the case of vector-valued functions with matrix weights, much of the ability to compare objects and dominate one by another is lost. Moreover, for a scalar weight $\omega$, we have the fact that a function $f\in L^p(\omega)$ if and only if $|f|\in L^p(\omega)$. Unfortunately, it is not true for matrix weights.\par

Based on these facts, different from the classical case, the available methods to prove Kolmogorov-Riesz's theorem seems to be not applicable in the setting of matrix weights. And we use a finite dimension argument without using the key lemma in \cite{HH} to obtain compactness criteria in matrix weighted Lebesgue spaces on $(\mathbb R^n, |\cdot|, m)$.\par

The paper is organized as follows. In Section \ref{s2}, we recall some basic notations and facts related to matrix weights. In particular, we recall the so-called John ellipsoid theorem which shows the existence of a positive-definite self-adjoint matrix for a given norm $\rho$ on $\mathbb C^d$; see Lemma \ref{lem-John} below.\par

Section \ref{s3} is devoted to the study of totally bounded sets in matrix weighted Lebesgue spaces on $(\mathbb R^n, |\cdot|, m)$. For a given family of norm $\rho=\{\rho_x\}_{x\in\mathbb R^n}$ on $\mathbb C^d$, we first apply the John ellipsoid theorem and establish a version of Kolmogorov-Riesz's theorem in $L^{p(\cdot)}(\rho)$ for an exponent function $p(\cdot)$. When $p\in (0,\infty)$ and $W$ is a matrix weight on $\mathbb{R}^n$ which is not necessarily invertible, we also obtain a Kolmogorov-Riesz theorem in $L^p(W)$ by following some idea from \cite{CMR}. As an application, a compactness criterion in degenerate Sobolev spaces with matrix weights is given.\par

In Section \ref{s4}, let $(X, d, \mu)$ be a proper metric measure space such that $\mu$ is continuous with respect to the metric $d$ and $\rho= \{\rho_x\}_{x\in X}$ be a family of norms on $\mathbb{C}^d$. We present a Kolmogorov-Riesz theorem in weighted Lebesgue spaces $L^p(\rho, \mu)$ with $p\in[1, \infty)$ on $(X, d, \mu)$ in terms of the average operator, and apply to $L^p(W, \mu)$ when $W$ is an invertible matrix weight on $\mathbb{R}^n$ and $\mu$ is continuous with respect to $|\cdot|$. We would like to mention that our method to prove Theorem \ref{thm-KRLrhov} (see also \cite[Theorem 5]{HH} and \cite[Theorem 3.1]{GZ}) relies on the translation invariance of the Euclidean metric and the Lebesgue measure, which fails on metric measure spaces.\par

In Section \ref{s5}, based on the result obtained in Section \ref{s4}, when the matrix weight $W$ is in the known $A_p$ class on $\mathbb R^n$ in \cite{R}, we further obtain an equivalent characterization of compact subsets in $L^p(W)$ for $p\in (1, \infty)$.\par

Throughout this paper, we will use the following notations. We always use $\omega(\cdot)$ to denote a scalar weight while $W(\cdot)$ to denote a matrix weight. Given two values $A$ and $B$, we will write $A\lesssim B$ if there exists a positive constant $c$, independent of appropriate quantities involved in $A$ and $B$, such that $A\leq cB$. We write $A\approx B$ if $A\lesssim B$ and $B\lesssim A$. We will use $p'$ to denote the conjugate exponent of $p$ when $p\in(1,\infty)$. For a given set $E$, $\chi_E$ means the characteristic function of $E$. Additionally, unless otherwise noted, $(\mathbb{R}^n,|\cdot|,m)$ is the underlying measure space.\par

\section{Preliminaries}\label{s2}
In this section, we recall some basic notations and facts about matrix weights; see \cite{G,CMR} and the references therein.\par

Let $\mathcal{M}_d$ denote the set of all complex-valued, $d\times d$ matrices. A matrix function on $\mathbb{R}^n$ is a map $W:\mathbb{R}^n\to \mathcal{M}_d$. We say that it is measurable if each component of $W$ is a measurable function, and invertible if $\det W(x)\not=0$ a.e. and so $W^{-1}$ exists. Let $\mathcal{S}_d$ be the set of all those $A\in \mathcal{M}_d$ that are self-adjoint and non-negative-definite. For each $A\in \mathcal{S}_d$, $A$ has $d$ non-negative real-valued eigenvalues $\lambda_i$, $1\leq i\leq d$, and the norm of $A$ is defined as the operator norm 
$$\|A\|_{op}:=\sup_{\mathbf{v}\in \mathbb{C}^d, |\mathbf{v}|=1}|A\mathbf{v}|=\max_{i}\lambda_i.$$ 
Moreover, there exists a unitary matrix $U$ such that $U^HAU$ is diagonal, where $U^H$ denotes the conjugate transpose matrix of $U$. We denote a diagonal matrix by $D(\lambda_1,\cdots,\lambda_d)=D(\lambda_i)$. For every $s>0$, we define $A^s:=UD(\lambda^s_i)U^H$. Furthermore, if $A$ is positive-definite, we set $A^{-s}:=UD(\lambda^{-s}_i)U^H$.\par

The following technical lemma is from \cite[Lemma 3.1]{CMR}.
\begin{lem}\label{lem-diag}
Given a measurable matrix function $W:\mathbb{R}^n\to \mathcal{S}_d$, there exists a $d\times d$ measurable matrix function $U$ defined on $\mathbb{R}^n$ such that $U^H(x)W(x)U(x)$ is diagonal, and $U(x)$ is unitary for every $x\in \mathbb{R}^n$.
\end{lem}
Based on Lemma \ref{lem-diag}, it is easy to see that for any measurable matrix function $W:\mathbb{R}^n\to \mathcal{S}_d$, $W^s$ is a measurable matrix function satisfying that, for any $x\in \mathbb{R}^n$,
\begin{align}\label{eq-2.1}
\|W^s(x)\|_{op}=\max\limits_i\lambda^s_i(x).\tag{2.1}
\end{align}
If $W$ is invertible, $W^{-s}$ is a measurable matrix function satisfying that, for any $x\in \mathbb{R}^n$,
\begin{align}\label{eq-2.2}
\|W^{-s}(x)\|^{-1}_{op}=\min\limits_i\lambda^s_i(x).\tag{2.2}
\end{align}\par

By a matrix weight on $\mathbb{R}^n$ we mean a measurable matrix function $W:\mathbb{R}^n\to \mathcal{S}_d$ such that $\|W\|_{op}\in L^1_{\rm{loc}}(\mathbb{R}^n)$. Equivalently, each eigenvalue function $\lambda_i\in L^1_{\rm{loc}}(\mathbb{R}^n),\ 1\leq i\leq d$. Define the matrix weighted Lebesgue space $L^p(W)$ for $p\in(0,\infty)$ to be the set of all measurable vector-valued functions $\mathbf{f}:=(f_1,\cdots,f_d)^T:\mathbb{R}^n\to \mathbb{C}^d$ such that $$\|\mathbf{f}\|^p_{L^p(W)}:=\int_{\mathbb{R}^n} \big|W^{\frac{1}{p}}(x)\mathbf{f}(x)\big|^p dx<\infty.$$\par

In many cases, it is more convenient to characterize matrix weighted Lebesgue spaces in the following language; see \cite{V}. Let $\rho:=\{\rho_x\}_{x\in \mathbb{R}^n}$ be a family of norms on $\mathbb{C}^d$, where for each $x\in \mathbb{R}^n$, $\rho_x:\mathbb{C}^d\to \mathbb{R}^+:=[0,\infty)$. Define the weighted Lebesgue space $L^p(\rho)$ for $p\in(0,\infty)$ to be the set of all measurable vector-valued functions $\mathbf{f}:\mathbb{R}^n\to \mathbb{C}^d$ such that $$\|\mathbf{f}\|^p_{L^p(\rho)}:=\int_{\mathbb{R}^n}[\rho_x(\mathbf{f}(x))]^pdx<\infty,$$ where we always assume that $\rho_x(\mathbf{f}(x))$ is a measurable function on $\mathbb{R}^n$ for any measurable vector-valued function $\mathbf{f}$.\par

For any given invertible matrix weight $W$, one can reduce $L^p(\rho)$ to $L^p(W)$ by setting $\rho_x(\cdot):=|W^{\frac{1}{p}}(x)\cdot|$. The following so-called John ellipsoid theorem (see \cite[Proposition 1.2]{G}) shows that the two matrix weighted Lebesgue spaces above actually coincide.
\begin{lem}\label{lem-John}
Given a norm $\rho$ on $\mathbb{C}^d$, there exists a positive-definite self-adjoint matrix $W$ such that
\begin{align}\label{eq-2.3}
\rho(\mathbf{v})\leq |W(\mathbf{v})|\leq d^{\frac{1}{2}} \rho(\mathbf{v}),\ \forall \ \mathbf{v}\in \mathbb{C}^d.\tag{2.3}
\end{align}
\end{lem}\par

We now recall the definition of matrix $A_p$ weights due to Frazier-Roudenko \cite{R,FR}; see also \cite{NT} for another definition of matrix $A_p$ weights when $p>1$.
\begin{df}\label{df-mw}
Let $W$ be an invertible matrix weight.
\begin{enumerate}
\item[$(\rm{i})$] When $p\in(1,\infty)$, we say $W\in A_p$ if $\|W^{-1}\|^{\frac{p'}{p}}_{op}\in L^1_{\rm{loc}}(\mathbb{R}^n)$ and
$$[W]_{A_p}:=\sup_{Q}\frac{1}{|Q|}\int_Q\bigg(\frac{1}{|Q|}\int_Q\big\|W^{\frac{1}{p}}(x)W^{-\frac{1}{p}}(y)\big\|^{p'}_{op}dy\bigg)^{\frac{p}{p'}}dx<\infty,$$ where the supremum is taken over all cubes $Q\subset\mathbb{R}^n$.
\item[$(\rm{ii})$] When $p\in(0,1]$, we say $W\in A_p$ if $\|W^{-1}\|_{op}\in L^1_{\rm{loc}}(\mathbb{R}^n)$ and
$$[W]_{A_p}:=\sup_{Q}\esssup_{x\in Q}\frac{1}{|Q|}\int_Q \big\|W^{\frac{1}{p}}(y)W^{-\frac{1}{p}}(x)\big\|^p_{op}dy<\infty,$$
where the first supremum is taken over all cubes $Q\subset\mathbb{R}^n$.
\end{enumerate}
\end{df}\par

We would like to mention that when $p\geq 1$, $d=1$ and $W(x)=\omega(x)$ is a scalar weight, the matrix $A_p$ condition is the Muckenhoupt $A_p$ condition. Moreover, we have the following lemma due to \cite[Lemma 4.5]{CMR}.
\begin{lem}\label{lem-ms}
Let $1<p<\infty$. If $W\in A_p$, then $\|W\|_{op}$ and $\|W^{-1}\|_{op}^{-1}$ are scalar $A_p$ weights.
\end{lem}\par

Next we recall a variant of the maximal operator introduced by Christ-Goldberg in \cite{CG,G}. The Christ-Goldberg maximal operator $M_{\omega}$ is defined as
$$M_{\omega}\mathbf{f}(x):=\sup_{B\ni x}\frac{1}{|B|}\int_B\big|W^{\frac{1}{p}}(x)W^{-\frac{1}{p}}(y)\mathbf{f}(y)\big|dy,$$
where the supremum is taken over all balls in $\mathbb{R}^n$ containing $x$. They obtained the following strong-type estimate for the Christ-Goldberg maximal operator.
\begin{lem}\label{lem-CGm}
Let $1<p<\infty$. If $W\in A_p$, then there exists $\delta>0$ such that when $q\in \{q>1:|p-q|<\delta\}$,
$$\|M_{\omega}\mathbf{f}\|_{L^q(\mathbb{R}^n)}\lesssim \|\mathbf{f}\|_{L^q(\mathbb{R}^n,\mathbb{C}^d)},\ \forall\ \mathbf{f}\in L^q(\mathbb{R}^n,\mathbb{C}^d),$$
where the implicit constant depends only on $q$ and
$$\|\mathbf{f}\|_{L^q(\mathbb{R}^n,\mathbb{C}^d)}:=\bigg(\int_{\mathbb{R}^n}|\mathbf{f}(x)|^qdx\bigg)^{\frac{1}{q}}.$$
\end{lem}

\section{Compactness criteria on $\mathbb{R}^n$}\label{s3}
This section is devoted to the study of Kolmogorov-Riesz's theorem in matrix weighted Lebesgue spaces on $\mathbb{R}^n$. In \cite[Theorem 5]{GM2}, G\'{o}rka-Macios established a compactness criterion in variable exponent Lebesgue spaces $L^{p(\cdot)}(\mathbb{R}^n)$. The first main result of this section is to obtain a generalization of \cite[Theorem 5]{GM2} to $L^{p(\cdot)}(\rho)$. Before that, we recall some basic notations and results about variable exponent Lebesgue spaces; see \cite{CF}.\par

Let $\mathcal{P}(\mathbb{R}^n)$ be the set of all measurable functions $p(\cdot):\mathbb{R}^n\to [1,\infty]$. The elements of $\mathcal{P}(\mathbb{R}^n)$ are called exponent functions. Given an exponent function $p(\cdot)\in \mathcal{P}(\mathbb{R}^n)$, we put $$p_+:=\esssup_{x\in \mathbb{R}^n}p(x),\ \ \ p_-:=\essinf_{x\in \mathbb{R}^n}p(x).$$
We assume that exponent functions $p(\cdot)$ are bounded, i.e. $p_+<\infty$. Define the weighted variable Lebesgue space $L^{p(\cdot)}(\rho)$ to be the set of all measurable vector-valued functions $\mathbf{f}:\mathbb{R}^n\to \mathbb{C}^d$ such that the modular $$\int_{\mathbb{R}^n}[\rho_x(\mathbf{f}(x))]^{p(x)}dx<\infty,$$ equipped with the Luxemburg norm $$\|\mathbf{f}\|_{L^{p(\cdot)}(\rho)}:=\inf\bigg\{\lambda>0:\int_{\mathbb{R}^n}\bigg[\frac{\rho_x(\mathbf{f}(x))}{\lambda}\bigg]^{p(x)}dx\leq 1\bigg\}.$$\par

According to the above definition and the convexity of the modular, we obtain the following useful results on the relationship between the modular and the norm.
\begin{lem}\label{lem-v}
Let $p_+<\infty$, $0<\lambda\leq 1$, and $\mathbf{f}\in L^{p(\cdot)}(\rho)$. Then the following statements are true:
\begin{enumerate}
\item[$(a)$] If $\|\mathbf{f}\|_{L^{p(\cdot)}(\rho)}\leq 1$, then $\int_{\mathbb{R}^n}[\rho_x(\mathbf{f}(x))]^{p(x)}dx \leq\|\mathbf{f}\|_{L^{p(\cdot)}(\rho)};$
\item[$(b)$] If $\|\mathbf{f}\|_{L^{p(\cdot)}(\rho)}> 1$, then $\int_{\mathbb{R}^n}[\rho_x(\mathbf{f}(x))]^{p(x)}dx \geq\|\mathbf{f}\|_{L^{p(\cdot)}(\rho)};$
\item[$(c)$] $\|\mathbf{f}\|_{L^{p(\cdot)}(\rho)}\leq \int_{\mathbb{R}^n}[\rho_x(\mathbf{f}(x))]^{p(x)}dx+1;$
\item[$(d)$] If $\int_{\mathbb{R}^n}[\rho_x(\mathbf{f}(x))]^{p(x)}dx\leq \lambda^{p_+}$, then $\|\mathbf{f}\|_{L^{p(\cdot)}(\rho)}\leq \lambda$.
\end{enumerate}
\end{lem}\par
\begin{proof}
$(a)$--$(c)$ hold by adapting the arguments in \cite[Corollary 2.22]{CF}. To prove $(d)$, noting that when $p_+<\infty$ and $0<\lambda\leq 1$, we have $$\int_{\mathbb{R}^n}\bigg[\frac{\rho_x(\mathbf{f}(x))}{\lambda}\bigg]^{p(x)}dx\leq \int_{\mathbb{R}^n}[\rho_x(\mathbf{f}(x))]^{p(x)}\lambda^{-p_+}dx\leq 1,$$ which implies $\|\mathbf{f}\|_{L^{p(\cdot)}(\rho)}\leq \lambda$.
\end{proof}

Based on Lemma \ref{lem-v}, we now present a Kolmogorov-Riesz theorem in $L^{p(\cdot)}(\rho)$.
\begin{thm}\label{thm-KRLrhov}
Let $p(\cdot)\in \mathcal{P}(\mathbb{R}^n)$, $p_+<\infty$ and $\rho:=\{\rho_x\}_{x\in \mathbb{R}^n}$ be a family of norms on $\mathbb{C}^d$ such that $W_x$ is an invertible matrix weight satisfying \eqref{eq-2.3} for every $x\in \mathbb{R}^n$ and $\|W_x\|^{p_+}_{op}\in L^1_{\rm{loc}}(\mathbb{R}^n)$. A subset $\mathcal{F}\subset L^{p(\cdot)}(\rho)$ is totally bounded if the following conditions are valid:
\begin{enumerate}
\item[$(a)$] $\mathcal{F}$ is bounded in the sense of the modular, that is, $$\sup_{\mathbf{f}\in \mathcal{F}}\int_{\mathbb{R}^n}[\rho_x(\mathbf{f}(x))]^{p(x)}dx<\infty;$$
\item[$(b)$] $\mathcal{F}$ uniformly vanishes at infinity, that is, $$\lim_{R\to \infty}\sup_{\mathbf{f}\in \mathcal{F}}\int_{B^c(0,R)}[\rho_x(\mathbf{f}(x))]^{p(x)}dx=0;$$
\item[$(c)$] $\mathcal{F}$ is equicontinuous, that is, $$\lim_{r\to 0}\sup_{\mathbf{f}\in \mathcal{F}}\sup_{y\in B(0,r)}\int_{\mathbb{R}^n}[\rho_x(\tau_y\mathbf{f}(x)-\mathbf{f}(x))]^{p(x)}dx=0.$$
\end{enumerate}
\end{thm}
\begin{proof}
Assume that $\mathcal{F}\subset L^{p(\cdot)}(\rho)$ satisfies $(a)$--$(c)$. Given $\epsilon>0$ small enough, to prove the total boundedness of $\mathcal{F}$, it suffices to find a finite $\epsilon$-net of $\mathcal{F}$. Denote by $R_i:=[-2^i,2^i)^n$ for $i\in \mathbb{Z}$. Then by condition $(b)$, there exists a positive integer $m$ large enough such that
\begin{equation}\label{eq-1}
\sup_{\mathbf{f}\in \mathcal{F}}\int_{\mathbb{R}^n}\big[\rho_x\big(\mathbf{f}(x)-\mathbf{f}(x)\chi_{R_m}(x)\big)\big]^{p(x)}dx<\epsilon.\tag{3.1}
\end{equation}
Moreover, by condition $(c)$, there exists an integer $t$ such that
\begin{equation}\label{eq-2}
\sup_{\mathbf{f}\in \mathcal{F}}\sup_{y\in R_t}\int_{\mathbb{R}^n}\big[\rho_x\big(\mathbf{f}(x-y)-\mathbf{f}(x)\big)\big]^{p(x)}dx<\epsilon.\tag{3.2}
\end{equation}\par

Let $\mathcal{Q}_i$, $i\in \mathbb{Z}$, be the family of dyadic cubes in $\mathbb{R}^n$, open on the right, whose vertices are adjacent points of the lattice $(2^i\mathbb{Z})^n$. 
For each $i\in\mathbb Z$, the cubes in $\mathcal{Q}_i$ are either disjoint or coincide. Thus there exists a sequence $\{Q_j\}_{j=1}^N$ of disjoint cubes in $\mathcal{Q}_t$ such that
\begin{small}$R_m=\bigcup\limits_{j=1}^NQ_j$\end{small},
where \begin{small}$N=2^{(m+1-t)n}$\end{small} is a positive integer.\par

For any $\mathbf{f}\in \mathcal{F}$ and $x\in \mathbb{R}^n$, define $$\Phi(\mathbf{f})(x):=
\begin{cases}
\mathbf{f}_{Q_j}:=\frac{1}{|Q_j|}\int_{Q_j}\mathbf{f}(y)dy, & x\in Q_j,\ j=1,\cdots,N,\\
0, & \text{otherwise}.
\end{cases}
$$
Then for every fixed $x\in \mathbb{R}^n$, by Lemma \ref{lem-John}, there exists a positive-definite self-adjoint matrix $W_x$ such that for each $j$,
\begin{align*}
\rho_x\big(\big(\mathbf{f}(x)-\mathbf{f}_{Q_j}\big)\chi_{Q_j}(x)\big)&\leq\big|W_x\big(\mathbf{f}(x)-\mathbf{f}_{Q_j}\big)\chi_{Q_j}(x)\big|\\
&= \bigg|\frac{1}{|Q_j|}\int_{Q_j}W_x\big(\mathbf{f}(x)-\mathbf{f}(y)\big)dy\chi_{Q_j}(x)\bigg|\\
&\lesssim \frac{1}{|Q_j|}\int_{Q_j}\big|W_x\big(\mathbf{f}(x)-\mathbf{f}(y)\big)\big|dy\chi_{Q_j}(x)\\
&\lesssim \frac{1}{|Q_j|}\int_{Q_j}\rho_x\big(\mathbf{f}(x)-\mathbf{f}(y)\big)dy\chi_{Q_j}(x),
\end{align*}
where the implicit constant depends only on $d$. Then from the choices of $\{Q_j\}^N_{j=1}$, the Jensen inequality, the Fubini theorem and \eqref{eq-2}, it follows that
\begin{align}\label{eq-3}
&\int_{\mathbb{R}^n}\big[\rho_x\big(\mathbf{f}(x)\chi_{R_m}(x)-\Phi(\mathbf{f})(x)\big)\big]^{p(x)}dx\notag\\
&\quad\lesssim \sum_{j=1}^N\int_{Q_j}\bigg|\frac{1}{|Q_j|}\int_{Q_j}\rho_x\big(\mathbf{f}(x)-\mathbf{f}(y)\big)dy\bigg|^{p(x)}dx\notag\\
&\quad\lesssim \sum_{j=1}^N\frac{1}{|Q_j|}\int_{Q_j}\int_{Q_j}\big[\rho_x\big(\mathbf{f}(x)-\mathbf{f}(y)\big)\big]^{p(x)}dx\ dy\notag\\
&\quad\approx 2^{-nt}\sum_{j=1}^N \int_{Q_j}\int_{Q_j}\big[\rho_x\big(\mathbf{f}(x)-\mathbf{f}(y)\big)\big]^{p(x)}dy\ dx\notag\\
&\quad\approx 2^{-nt}\sum_{j=1}^N \int_{Q_j}\int_{x-Q_j}\big[\rho_x\big(\mathbf{f}(x)-\mathbf{f}(x-y)\big)\big]^{p(x)}dy\ dx\notag\\
&\quad\lesssim 2^{-nt}\int_{\mathbb{R}^n}\int_{R_t}\big[\rho_x\big(\mathbf{f}(x)-\mathbf{f}(x-y)\big)\big]^{p(x)}dy\ dx\notag\\
&\quad\approx 2^{-nt}\int_{R_t}\int_{\mathbb{R}^n}\big[\rho_x\big(\mathbf{f}(x)-\mathbf{f}(x-y)\big)\big]^{p(x)}dx\ dy\notag\\
&\quad\lesssim 2^{-nt}|R_t| \sup_{y\in R_t}\int_{\mathbb{R}^n}\big[\rho_x\big(\mathbf{f}(x)-\mathbf{f}(x-y)\big)\big]^{p(x)}dx\lesssim 2^n\epsilon,\tag{3.3}
\end{align}
where we use the fact that $x-Q_j:=\{x-y:y\in Q_j\}\subset R_t$ when $x\in Q_j$. Note that
\begin{align*}
&\int_{\mathbb{R}^n}\big[\rho_x\big(\mathbf{f}(x)-\Phi(\mathbf{f})(x)\big)\big]^{p(x)}dx\\
&\quad= \bigg[\int_{\mathbb{R}^n\backslash R_m}+\int_{R_m}\bigg]\big[\rho_x\big(\mathbf{f}(x)-\Phi(\mathbf{f})(x)\big)\big]^{p(x)}dx\\
&\quad= \int_{\mathbb{R}^n}\big[\rho_x\big(\mathbf{f}(x)-\mathbf{f}(x)\chi_{R_m}(x)\big)\big]^{p(x)}dx+ \int_{\mathbb{R}^n}\big[\rho_x\big(\mathbf{f}(x)\chi_{R_m}(x)-\Phi(\mathbf{f})(x)\big)\big]^{p(x)}dx.
\end{align*}
This via \eqref{eq-1} and \eqref{eq-3} implies that
\begin{align}\label{eq-11}
\sup_{\mathbf{f}\in \mathcal{F}}\int_{\mathbb{R}^n}\big[\rho_x\big(\mathbf{f}(x)-\Phi(\mathbf{f})(x)\big)\big]^{p(x)}dx\lesssim \epsilon,\tag{3.4}
\end{align}
where the implicit constant depends only on $n$ and $d$. Since $p_+< \infty$, then from \eqref{eq-11} and Lemma \ref{lem-v}, it suffices to show that $\Phi(\mathcal{F})$ is totally bounded in $L^{p(\cdot)}(\rho)$.\par

Note that by condition $(a)$, we have
\begin{align*}
&\sup_{\mathbf{f}\in \mathcal{F}}\int_{\mathbb{R}^n}\big[\rho_x\big(\Phi(\mathbf{f})(x)\big)\big]^{p(x)}dx\\
&\quad\leq 2^{p_+-1}\bigg(\sup_{\mathbf{f}\in \mathcal{F}}\int_{\mathbb{R}^n}\big[\rho_x\big(\mathbf{f}(x)-\Phi(\mathbf{f})(x)\big)\big]^{p(x)}dx+ \sup_{\mathbf{f}\in \mathcal{F}}\int_{\mathbb{R}^n}\big[\rho_x\big(\mathbf{f}(x)\big)\big]^{p(x)}dx\bigg)<\infty.
\end{align*}
Then by \cite[Remark 2.10]{CF}, it follows that for any $\mathbf{f}\in \mathcal{F}$, $$|W_x\Phi(\mathbf{f})(x)|\leq d^{\frac{1}{2}}\rho_x\big(\Phi(\mathbf{f})(x)\big)<\infty\text{\ \ a.e.}\ x\in \mathbb{R}^n.$$ Since $W_x$ is positive-definite for every $x\in \mathbb{R}^n$, we obtain $$|\Phi(\mathbf{f})(x)|<\infty\text{\ \ a.e.}\ x\in \mathbb{R}^n,$$ which implies $|\mathbf{f}_{Q_j}|<\infty,\ j=1,\cdots,N.$ From this and $\|W_x\|^{p_+}_{op}\in L^1_{\rm{loc}}(\mathbb{R}^n)$, we see that $\Phi$ is a map from $\mathcal{F}$ to $\mathcal{B}$, a finite dimensional Banach subspace of $L^{p(\cdot)}(\rho)$. Notice that $\Phi({\mathcal{F}})\subset \mathcal{B}$ is bounded, and hence is totally bounded. The proof of Theorem \ref{thm-KRLrhov} is complete.
\end{proof}\par

As a corollary of Theorem \ref{thm-KRLrhov}, by taking $p(\cdot)\equiv p\in [1,\infty)$,
we have the following Kolmogorov-Riesz theorem in $L^p(\rho)$ defined in Section \ref{s2}.
\begin{cor}\label{cor-KRLrho}
Let $1\leq p<\infty$, $\rho:=\{\rho_x\}_{x\in \mathbb{R}^n}$ be a family of norms on $\mathbb{C}^d$ such that $W_x$ is an invertible matrix weight satisfying \eqref{eq-2.3} for every $x\in \mathbb{R}^n$ and $\|W_x\|^p_{op}\in L^1_{\rm{loc}}(\mathbb{R}^n)$. A subset $\mathcal{F}$ of $L^p(\rho)$ is totally bounded if the following conditions hold:
\begin{enumerate}
\item[$(a)$] $\mathcal{F}$ is bounded, i.e. $\sup\limits_{\mathbf{f}\in \mathcal{F}}\|\mathbf{f}\|_{L^p(\rho)}<\infty;$
\item[$(b)$] $\mathcal{F}$ uniformly vanishes at infinity, that is, $$\lim_{R\to \infty}\sup_{\mathbf{f}\in \mathcal{F}}\|\mathbf{f}\chi_{B^c(0,R)}\|_{L^p(\rho)}=0;$$
\item[$(c)$] $\mathcal{F}$ is equicontinuous, that is, $$\lim_{r\to 0}\sup_{\mathbf{f}\in \mathcal{F}}\sup_{y\in B(0,r)}\|\tau_y\mathbf{f}-\mathbf{f}\|_{L^p(\rho)}=0.$$
\end{enumerate}
\end{cor}\par

As an application of Corollary \ref{cor-KRLrho}, by setting $\rho_x(\cdot):=|W^{\frac{1}{p}}(x)\cdot|$ and \eqref{eq-2.1}, we have the following compactness criterion in $L^p(W)$ for $p\in[1,\infty)$.
\begin{cor}\label{cor-mwKRL}
Let $1\leq p<\infty$, $W$ be an invertible matrix weight. A subset $\mathcal{F}$ of $L^p(W)$ is totally bounded if the following conditions hold:
\begin{enumerate}
\item[$(a)$] $\mathcal{F}$ is bounded, i.e. $\sup\limits_{\mathbf{f}\in \mathcal{F}}\|\mathbf{f}\|_{L^p(W)}<\infty;$
\item[$(b)$] $\mathcal{F}$ uniformly vanishes at infinity, that is, $$\lim_{R\to \infty}\sup_{\mathbf{f}\in \mathcal{F}}\|\mathbf{f}\chi_{B^c(0,R)}\|_{L^p(W)}=0;$$
\item[$(c)$] $\mathcal{F}$ is equicontinuous, that is, $$\lim_{r\to 0}\sup_{\mathbf{f}\in \mathcal{F}}\sup_{y\in B(0,r)}\|\tau_y\mathbf{f}-\mathbf{f}\|_{L^p(W)}=0.$$
\end{enumerate}
\end{cor}

\begin{rk}\label{rk-KRLrho}
It is worth noting that compared with the compactness criteria in \cite{CC,XYY,COY}, there is no additional assumption on $W$. Based on this, from the fact that matrix weighted Lebesgue spaces are not translation invariant, on which the translation operator $\tau_y$ is not continuous, it follows that Corollary \ref{cor-mwKRL} is a strong sufficient condition for precompactness in $L^p(W)$.
\end{rk}\par

When $p\in (0,\infty)$ and $W$ is a matrix weight which is not necessarily invertible, we also have the following characterization of totally bounded sets in $L^p(W)$, which extends Theorem \ref{thm-GZ} to the setting of matrix weights.
\begin{thm}\label{thm-mwKRL}
Let $0<p<\infty$, $W$ be a matrix weight. A subset $\mathcal{F}$ of $L^p(W)$ is totally bounded if it satisfies the conditions $(a),(b)$ in Corollary \ref{cor-mwKRL} and
\begin{enumerate}
\item[$(c^*)$] $\mathcal{F}$ is equicontinuous in the sense that $$\lim_{r\to 0}\sup_{\mathbf{f}\in \mathcal{F}}\sup_{y\in B(0,r)}\bigg(\int_{\mathbb{R}^n}\Big|D^{\frac{1}{p}}(x)\Big(U^H(x-y)\mathbf{f}(x-y)-U^H(x)\mathbf{f}(x)\Big)\Big|^pdx\bigg)^{\frac{1}{p}}=0.$$
\end{enumerate}
Here, $D^{\frac{1}{p}}:=U^HW^{\frac{1}{p}}U$ is a matrix weight by Lemma \ref{lem-diag}.
\end{thm}
\begin{proof}
We use some ideas from \cite[Proposition 3.6]{CMR}. Assume that $\mathcal{F}\subset L^p(W)$ satisfies $(a)$--$(c)$. For any $\mathbf{f}\in L^p(W)$, denote by $\tilde{\mathbf{f}}:=U^H\mathbf{f}$ and $\tilde{\mathcal{F}}:=\{\tilde{\mathbf{f}}\}_{\mathbf{f}\in \mathcal{F}}$. Then condition $(c^*)$ is equivalent to the  equicontinuity of $\tilde{\mathcal{F}}\subset L^p(D)$, that is, $$\lim_{r\to 0}\sup_{\tilde{\mathbf{f}}\in \tilde{\mathcal{F}}}\sup_{y\in B(0,r)}\|\tau_y\tilde{\mathbf{f}}-\tilde{\mathbf{f}}\|_{L^p(D)}=0.$$
Moreover, from the orthogonality of $U$, it follows that $$\big|D^{\frac{1}{p}}\tilde{\mathbf{f}}\big|=\big|U^HW^{\frac{1}{p}}UU^H\mathbf{f}\big|=\big|W^{\frac{1}{p}}\mathbf{f}\big|.$$
This via $(a)$ and $(b)$ for $\mathcal{F}$ shows that $\tilde{\mathcal{F}}\subset L^p(D)$ is also bounded and uniformly vanishes at infinity. Observe that  if $\tilde{\mathcal{F}}\subset L^p(D)$ is totally bounded, so does $\mathcal{F}\subset L^p(W)$ (indeed, the converse is also true). It suffices to verify the total boundedness of $\tilde{\mathcal{F}}\subset L^p(D)$, that is, to find a finite $\epsilon$-net of $\tilde{\mathcal{F}}$ for each fixed $\epsilon>0$.\par

Note that for every $1\leq i\leq d$, we have $|D^{\frac{1}{p}}\tilde{\mathbf{f}}|=|D(\lambda^{\frac{1}{p}}_i)\tilde{\mathbf{f}}|\geq \lambda^{\frac{1}{p}}_i|\tilde{f}_i|,$ which implies that $\{\tilde{f}_i\}_{\mathbf{f}\in \mathcal{F}}\subset L^p(\lambda_i)$ satisfies conditions $(a)$--$(c)$ of Theorem \ref{thm-GZ} with $\omega$ replaced by $\lambda_i$, where $\lambda_i\in L^1_{\rm{loc}}(\mathbb{R}^n)$ is a nonnegative eigenvalue function and $\tilde{\mathbf{f}}:=(\tilde{f}_1,\cdots,\tilde{f}_d)^T$. Then we obtain that $\{\tilde{f}_i\}_{\mathbf{f}\in \mathcal{F}}\subset L^p(\lambda_i)$ is totally bounded by Theorem \ref{thm-GZ}.\par

Hence, given $\epsilon>0$, for every $\mathbf{f}\in \mathcal{F}$ and $1\leq i\leq d$, there exists $g_i\in L^p(\lambda_i)$ such that $\|\tilde{f}_i-g_i\|_{L^p(\lambda_i)}<\epsilon$. Let $\mathbf{g}:=(g_1,\cdots,g_d)^T$. Then by the equivalence of norms in $\mathbb{C}^d$ and our choice of the $g_i$'s, we have
\begin{align*}
\big\|\tilde{\mathbf{f}}-\mathbf{g}\big\|_{L^p(D)}\approx \sum^d_{i=1} \|\tilde{f}_i-g_i\|_{L^p(\lambda_i)}\lesssim \epsilon,
\end{align*}
where implicit constants depend only on $p$ and $d$. Finally, by the total boundedness of $\{\tilde{f}_i\}_{\mathbf{f}\in \mathcal{F}}$, we conclude that $\{\mathbf{g}\}\subset L^p(D)$ is a finite $\epsilon$-net of $\tilde{\mathcal{F}}$. This completes the proof of Theorem \ref{thm-mwKRL}.
\end{proof}\par

\begin{rk}\label{rk-mwKRL}
When $d=1$ and $W(x)=\omega(x)$, we have $D^{\frac{1}{p}}(x)=\omega^{\frac{1}{p}}(x)$ and $U(x)=1$. In this case, both Theorem \ref{thm-mwKRL} $(c^*)$ and Corollary \ref{cor-mwKRL} $(c)$ become Theorem \ref{thm-GZ} $(c)$.
\end{rk}

Finally, we end this section with an application in degenerate Sobolev spaces with matrix weights. Let $W$ be an invertible matrix weight and set $v:=\|W\|_{op}$. For $p\in [1,\infty)$, the degenerate Sobolev space $\mathcal{W}^{1,p}_W(\mathbb{R}^n)$ due to Cruz-Uribe--Moen--Rodney \cite{CMR} is defined as the set of all $f\in \mathcal{W}^{1,1}_{\rm{loc}}(\mathbb{R}^n)$ such that
\begin{equation*}
\|f\|_{\mathcal{W}^{1,p}_W(\mathbb{R}^n)}:=\|f\|_{L^p(v)}+\|\nabla f\|_{L^p(W)}<\infty,
\end{equation*}
where $\nabla f$ is the gradient of $f$. The degenerate Sobolev space $\mathcal{W}^{1,p}_W(\mathbb{R}^n)$ is an extension of scalar weighted Sobolev spaces. From Theorem \ref{thm-GZ} and Corollary \ref{cor-mwKRL}, we have the following compactness criterion in $\mathcal{W}^{1,p}_W(\mathbb{R}^n)$, which is an extension of \cite[Corollary 9]{HH} and \cite[Theorem 12]{AU}.

\begin{cor}\label{cor-s}
A subset $\mathcal{F}\subset \mathcal{W}^{1,p}_W(\mathbb{R}^n)$ is totally bounded if the following conditions hold:
\begin{enumerate}
\item[$(a)$] $\mathcal{F}$ is bounded, i.e. $\sup\limits_{f\in \mathcal{F}}\|f\|_{\mathcal{W}^{1,p}_W(\mathbb{R}^n)}<\infty;$
\item[$(b)$] $\mathcal{F}$ uniformly vanishes at infinity, that is, $$\lim_{R\to \infty}\sup_{f\in \mathcal{F}}\|f\chi_{B^c(0,R)}\|_{\mathcal{W}^{1,p}_W(\mathbb{R}^n)}=0;$$
\item[$(c)$] $\mathcal{F}$ is equicontinuous, that is, $$\lim_{r\to 0}\sup_{f\in \mathcal{F}}\sup_{y\in B(0,r)}\|\tau_yf-f\|_{\mathcal{W}^{1,p}_W(\mathbb{R}^n)}=0.$$
\end{enumerate}
\end{cor}
\begin{proof}
First, note that $\mathcal{F}\subset \mathcal{W}^{1,p}_W(\mathbb{R}^n)$ satisfies $(a)$--$(c)$ if and only if $\mathcal{F}\subset L^p(v)$ satisfies $(a)$--$(c)$ of Theorem \ref{thm-GZ} and $\nabla(\mathcal{F}):=\{\nabla f\}_{f\in \mathcal{F}}\subset L^p(W)$ satisfies $(a)$--$(c)$ of Corollary \ref{cor-mwKRL}. Then since $W$ is an invertible matrix weight, by Theorem \ref{thm-GZ} and Corollary \ref{cor-mwKRL}, we obtain that both $\mathcal{F}\subset L^p(v)$ and $\nabla(\mathcal{F})\subset L^p(W)$ are totally bounded. Corollary \ref{cor-s} then follows from the fact that a set in metric spaces is totally bounded if and only if it is Cauchy-precompact, that is, every sequence admits a Cauchy subsequence; see \cite[p. 262]{W}.
\end{proof}

\section{Compactness criteria on metric measure spaces}\label{s4}
This section is devoted to the study of totally bounded sets in matrix weighted Lebesgue spaces on metric measure spaces. We begin with some basic facts about metric measure spaces in \cite{GG,GR}.\par

Let $(X,d,\mu)$ be a metric measure space equipped with a metric $d$ and a positive Borel regular measure $\mu$. Let $$B(x,r):=\{y\in X: d(x,y)<r\}$$ be the ball of the radius $r>0$ with center $x\in X$. We assume that the measure of every open nonempty set is strictly positive, and that the measure of every bounded set is finite.\par

\begin{df}\label{df-proper}
A metric space is proper if every closed bounded set is compact.
\end{df}

We remark that since every bounded set in a geometrically doubling metric space is totally bounded (see \cite[Lemma 2.3]{H}), a geometrically doubling metric space is proper if and only if it is complete via the Hausdorff criterion.
\begin{df}\label{df-mcm}
Let $(X,d,\mu)$ be a metric measure space. The measure $\mu$ is said to be continuous with respect to the metric $d$ if for any $x\in X$ and $r>0$ the following condition is valid: $$\lim_{y\to x}\mu[B(x,r)\Delta B(y,r)]=0,$$ where $A\Delta B$ stands for the symmetric difference of sets $A$ and $B$. We call such a measure metrically continuous for short, when no confusions can rise.
\end{df}

From Definition \ref{df-mcm}, we have the following lemma.
\begin{lem}\label{lem-mcmlup}
If $\mu$ is metrically continuous, then for every compact set $K\subset X$ and $r>0$, $$\inf_{x\in K}\mu[B(x,r)]>0.$$
\end{lem}
\begin{proof}
First, since $\mu$ is metrically continuous, then from Definition \ref{df-mcm}, we deduce that for any $x,y\in X$ and $r>0$,
$$\big|\mu[B(x,r)]-\mu[B(y,r)]\big|\leq \mu[B(x,r)\Delta B(y,r)],$$
which implies that for any given $r>0$, the map $x\mapsto \mu[B(x,r)]$ is continuous. Lemma \ref{lem-mcmlup} then follows from the extreme value theorem.
\end{proof}\par

We now recall a vector-valued version of the classical Arzel\'{a}-Ascoli theorem; see \cite[Lemma 2.1]{GN} and \cite[Theorem 2]{HH}.
\begin{lem}\label{lem-AA}
Let $K$ be a compact topological space and $C(K,\mathbb{C}^d)$ be the space of $\mathbb{C}^d$-valued continuous functions on $K$ with the topology of uniform convergence. A subset $\mathcal{F}$ of $C(K,\mathbb{C}^d)$ is totally bounded if and only if the following conditions are valid:
\begin{enumerate}
\item[$(a)$] $\mathcal{F}$ is pointwise bounded, i.e. $\sup\limits_{\mathbf{f}\in \mathcal{F}}|\mathbf{f}(x)|<\infty,\ \forall\ x\in K;$
\item[$(b)$] $\mathcal{F}$ is equicontinuous, that is, for every $x\in K$ and $\epsilon >0$, there is a neighborhood $U$ of $x$ such that $$|\mathbf{f}(x)-\mathbf{f}(y)|<\epsilon,\ \forall\ y\in U,\ \mathbf{f}\in \mathcal{F}.$$
\end{enumerate}
\end{lem}\par

Now we give some necessary definitions and notations of matrix weights on metric measure spaces. A matrix function on $X$ is a map $W:X\to \mathcal{M}_d$. We say that it is $\mu$-measurable if each component of $W$ is a $\mu$-measurable function on $X$, and invertible if $\det W(x)\not=0$ $\mu$-a.e. and so $W^{-1}$ exists.\par

By a matrix weight on $X$ we mean a $\mu$-measurable matrix function $W:X\to \mathcal{S}_d$ such that $\|W\|_{op}\in L^1_{\rm{loc}}(X,d,\mu)$. Equivalently, each eigenvalue function $\lambda_i\in L^1_{\rm{loc}}(X,d,\mu),\ 1\leq i\leq d$. Let $\rho:=\{\rho_x\}_{x\in X}$ be a family of norms on $\mathbb{C}^d$, where for each $x\in X$, $\rho_x:\mathbb{C}^d\to \mathbb{R}^+$. Define the weighted Lebesgue space $L^p(\rho,\mu)$ for $p\in(0,\infty)$ be the class of all $\mu$-measurable vector-valued functions $\mathbf{f}:X\to \mathbb{C}^d$ such that
$$\|\mathbf{f}\|^p_{L^p(\rho,\mu)}:=\int_X[\rho_x(\mathbf{f}(x))]^pd\mu(x)<\infty,$$
where we always assume that $\rho_x(\mathbf{f}(x))$ is a $\mu$-measurable function on $X$ for any $\mu$-measurable vector-valued function $\mathbf{f}$.\par

Now we  present our main result in this section as follows, in which we replace the translation operator by the average operator and apply Lemma \ref{lem-AA}, inspired by \cite[Theorem 3.1]{GR}.
\begin{thm}\label{thm-KRrho}
Assume that $(X,d,\mu)$ is a proper metric measure space with a metrically continuous measure $\mu$. Let $1<p<\infty$, $\rho:=\{\rho_x\}_{x\in X}$ be a family of norms on $\mathbb{C}^d$ such that $W_x$ is an invertible matrix weight satisfying \eqref{eq-2.3} for every $x\in X$ and $\|W_x\|^p_{op},\|W^{-1}_x\|^{p'}_{op}\in L^1_{\rm{loc}}(X,d,\mu)$. A subset $\mathcal{F}$ of $L^p(\rho,\mu)$ is totally bounded if the following conditions hold:
\begin{enumerate}
\item[$(a)$] $\mathcal{F}$ is bounded, i.e. $\sup\limits_{\mathbf{f}\in \mathcal{F}}\|\mathbf{f}\|_{L^p(\rho,\mu)}<\infty;$
\item[$(b)$] $\mathcal{F}$ uniformly vanishes at infinity, that is, for some $x_0\in X$, $$\lim_{R\to \infty}\sup_{\mathbf{f}\in \mathcal{F}}\|\mathbf{f}\chi_{X\backslash B(x_0,R)}\|_{L^p(\rho,\mu)}=0;$$
\item[$(c)$] $\mathcal{F}$ is equicontinuous, that is, $$\lim_{r\to 0}\sup_{\mathbf{f}\in \mathcal{F}}\|S_r\mathbf{f}-\mathbf{f}\|_{L^p(\rho,\mu)}=0.$$
\end{enumerate}
Here, $S_r$ denotes the average operator: $$S_r\mathbf{f}(x):=\frac{1}{\mu[B(x,r)]}\int_{B(x,r)}\mathbf{f}(y)\ d\mu(y).$$
\end{thm}

\begin{proof}
Assume that $\mathcal{F}\subset L^p(\rho,\mu)$ satisfies $(a)$--$(c)$. Given $\epsilon>0$, to prove the total boundedness of $\mathcal{F}$, it suffices to find a finite $\epsilon$-net of $\mathcal{F}$. By condition $(b)$, there exists $R>0$ such that
\begin{equation}\label{eq-4}
\sup_{\mathbf{f}\in \mathcal{F}}\|\mathbf{f}-\mathbf{f}\chi_{B(x_0,R)}\|_{L^p(\rho,\mu)}<\frac{\epsilon}{3}.\tag{4.1}
\end{equation}
Moreover, by condition $(c)$, there exists $r\in (0,R)$ such that
\begin{equation}\label{eq-5}
\sup_{\mathbf{f}\in \mathcal{F}}\|S_r\mathbf{f}-\mathbf{f}\|_{L^p(\rho,\mu)}<\frac{\epsilon}{3}.\tag{4.2}
\end{equation}
Then by \eqref{eq-4} and \eqref{eq-5}, we have
\begin{align*}
&\sup_{\mathbf{f}\in \mathcal{F}}\|(S_r\mathbf{f})\chi_{B(x_0,R)}-\mathbf{f}\|_{L^p(\rho,\mu)}\\
&\quad \leq \sup_{\mathbf{f}\in \mathcal{F}}\|(S_r\mathbf{f})\chi_{B(x_0,R)}-\mathbf{f}\chi_{B(x_0,R)}\|_{L^p(\rho,\mu)}+ \sup_{\mathbf{f}\in \mathcal{F}}\|\mathbf{f}-\mathbf{f}\chi_{B(x_0,R)}\|_{L^p(\rho,\mu)}<\frac{2\epsilon}{3}.
\end{align*}
So we only need to show that $\{(S_r\mathbf{f})\chi_{B(x_0,R)}\}_{\mathbf{f}\in \mathcal{F}}$ has a finite $\frac{\epsilon}{3}$-net.\par
Next, we turn to verify that $\{S_r\mathbf{f}\}_{\mathbf{f}\in \mathcal{F}}$ is pointwise bounded and equicontinuous on $\bar{B}(x_0,R)$, where $$\bar{B}(x_0,R):=\{y\in X:d(x_0,y)\leq R\}$$
is a closed bounded subset of $X$, and hence is compact by Definition \ref{df-proper}. From Lemma \ref{lem-John}, we have the following estimate for any $\mathbf{f}\in L^p(\rho,\mu)$,
\begin{align}\label{eq-6}
&\int_{B(x_0,2R)}|\mathbf{f}(y)|d\mu(y)\notag\\
&\quad =\int_{B(x_0,2R)}|W_y^{-1}W_y\mathbf{f}(y)|d\mu(y)\leq \int_{B(x_0,2R)}\|W_y^{-1}\|_{op}|W_y\mathbf{f}(y)|d\mu(y)\notag\\
&\quad\leq \bigg(\int_{B(x_0,2R)}\|W_y^{-1}\|^{p'}_{op}d\mu(y)\bigg)^{\frac{1}{p'}}
\bigg(\int_{B(x_0,2R)}|W_y\mathbf{f}(y)|^pd\mu(y)\bigg)^{\frac{1}{p}}\notag\\
&\quad\approx \bigg(\int_{B(x_0,2R)}\|W_y^{-1}\|^{p'}_{op}d\mu(y)\bigg)^{\frac{1}{p'}}
\bigg(\int_{B(x_0,2R)}\big[\rho_y(\mathbf{f}(y))\big]^pd\mu(y)\bigg)^{\frac{1}{p}}\notag\\
&\quad\lesssim \bigg(\int_{B(x_0,2R)}\|W_y^{-1}\|^{p'}_{op}d\mu(y)
\bigg)^{\frac{1}{p'}}\|\mathbf{f}\|_{L^p(\rho,\mu)}.\tag{4.3}
\end{align}
It follows that for any $\mathbf{f}\in \mathcal{F}$ and any fixed $x\in \bar{B}(x_0,R)$,
\begin{align}\label{eq-7}
|S_r\mathbf{f}(x)|&\lesssim \frac{1}{\mu[B(x,r)]}\int_{B(x,r)}|\mathbf{f}(y)|d\mu(y)\notag\\
&\lesssim \frac{1}{\mu[B(x,r)]}\int_{B(x_0,2R)}|\mathbf{f}(y)|d\mu(y)\notag\\
&\lesssim \frac{\sup\limits_{\mathbf{f}\in \mathcal{F}}\|\mathbf{f}\|_{L^p(\rho,\mu)}}{\mu[B(x,r)]}\bigg(\int_{B(x_0,2R)}\|W_y^{-1}\|^{p'}_{op}d\mu(y)\bigg)^{\frac{1}{p'}},\tag{4.4}
\end{align}
where we use the fact that $B(x,r)\subset B(x_0,2R)$ when $x\in \bar{B}(x_0,R)$. Since $\|W^{-1}_y\|^{p'}_{op}\in L^1_{\rm{loc}}(X,d,\mu)$, then by condition $(a)$, we obtain that $\{S_r\mathbf{f}\}_{\mathbf{f}\in \mathcal{F}}$ is pointwise bounded on $\bar{B}(x_0,R)$.\par

Furthermore, for any $\mathbf{f}\in \mathcal{F}$ and any fixed $x\in \bar{B}(x_0,R)$, by \eqref{eq-7}, we have the following estimate that for any $y\in \bar{B}(x_0,R)$,
\begin{align}\label{eq-8}
&|S_r\mathbf{f}(y)-S_r\mathbf{f}(x)|\notag\\
&\quad\leq \bigg|\frac{1}{\mu[B(y,r)]}\int_{B(y,r)}\mathbf{f}(z)d\mu(z)-\frac{1}{\mu[B(x,r)]}\int_{B(y,r)}\mathbf{f}(z)d\mu(z)\bigg|\notag\\
&\qquad+ \bigg|\frac{1}{\mu[B(x,r)]}\int_{B(y,r)}\mathbf{f}(z)d\mu(z)-\frac{1}{\mu[B(x,r)]}\int_{B(x,r)}\mathbf{f}(z)d\mu(z)\bigg|\notag\\
&\quad\lesssim \frac{\mu[B(x,r)\Delta B(y,r)]}{\mu[B(x,r)]\mu[B(y,r)]}\int_{B(y,r)}|\mathbf{f}(z)|d\mu(z)+ \frac{1}{\mu[B(x,r)]}\int_{B(x,r)\Delta B(y,r)}|\mathbf{f}(z)|d\mu(z)\notag\\
&\quad\lesssim \frac{\mu[B(x,r)\Delta B(y,r)]}{\mu[B(x,r)]\mu[B(y,r)]}\int_{B(x_0,2R)}|\mathbf{f}(z)|d\mu(z)+ \frac{1}{\mu[B(x,r)]}\int_{B(x,r)\Delta B(y,r)}|\mathbf{f}(z)|d\mu(z)\notag\\
&\quad\lesssim \frac{\mu[B(x,r)\Delta B(y,r)]}{\mu[B(x,r)]\mu[B(y,r)]}\sup_{\mathbf{f}\in \mathcal{F}}\|\mathbf{f}\|_{L^p(\rho,\mu)}\bigg(\int_{B(x_0,2R)}\|W_z^{-1}\|^{p'}_{op}d\mu(z)\bigg)^{\frac{1}{p'}}\notag\\
&\qquad+ \frac{1}{\mu[B(x,r)]}\int_{B(x,r)\Delta B(y,r)}|\mathbf{f}(z)|d\mu(z).\tag{4.5}
\end{align}
Since $\mu$ is metrically continuous, then by \eqref{eq-6}, \eqref{eq-7} and Lemma \ref{lem-mcmlup}, for any $0<h\leq r$, $x,y\in \bar{B}(x_0,R)$ and $\mathbf{f}\in \mathcal{F}$, a direct calculation yields that
\begin{align}\label{eq-9}
&\int_{B(x,r)\Delta B(y,r)}|\mathbf{f}(z)|d\mu(z)\notag\\
&\quad\leq \int_{B(x,r)\Delta B(y,r)}|\mathbf{f}(z)-S_h\mathbf{f}(z)|d\mu(z)+\int_{B(x,r)\Delta B(y,r)}|S_h\mathbf{f}(z)|d\mu(z)\notag\\
&\quad\leq \int_{B(x_0,2R)}|\mathbf{f}(z)-S_h\mathbf{f}(z)|d\mu(z)+\int_{B(x,r)\Delta B(y,r)}|S_h\mathbf{f}(z)|d\mu(z)\notag\\
&\quad\lesssim \bigg(\int_{B(x_0,2R)}\|W_z^{-1}\|^{p'}_{op}d\mu(z)\bigg)^{\frac{1}{p'}}\sup_{\mathbf{f}\in \mathcal{F}}\|S_h\mathbf{f}-\mathbf{f}\|_{L^p(\rho,\mu)}\notag\\
&\qquad+ \frac{\mu[B(x,r)\Delta B(y,r)]}{\inf\limits_{z\in \bar{B}(x_0,2R)}\mu[B(z,h)]}\bigg(\int_{B(x_0,3R)}\|W_z^{-1}\|^{p'}_{op}d\mu(z)\bigg)^{\frac{1}{p'}}\sup_{\mathbf{f}\in \mathcal{F}}\|\mathbf{f}\|_{L^p(\rho,\mu)}\tag{4.6},
\end{align}
where we use the fact that $B(z,h)\subset B(x_0,3R)$ when $z\in B(x_0,2R)$, and implicit constants depend only on $d$. Then by the arbitrariness of $h$, condition $(c)$, \eqref{eq-8} and \eqref{eq-9}, we obtain that $\{S_r\mathbf{f}\}_{\mathbf{f}\in \mathcal{F}}$ is equicontinuous on $\bar{B}(x_0,R)$.\par

So from Lemma \ref{lem-AA} we conclude that $\{S_r\mathbf{f}\}_{\mathbf{f}\in \mathcal{F}}$ is totally bounded in $C(\bar{B}(x_0,r),\mathbb{C}^d)$. It follows that there exists $\{\mathbf{f}_k\}^N_{k=1}\subset\mathcal{F}$ such that $\{S_r\mathbf{f}_k\}^N_{k=1}$ is an $\frac{\epsilon}{A}$-net of $\{S_r\mathbf{f}\}_{\mathbf{f}\in \mathcal{F}}$ for given $\epsilon$, where \begin{small}$A:=3\big(\int_{B(x_0,R)}\|W_x\|^p_{op}d\mu(x)\big)^{\frac{1}{p}}$\end{small}.\par

Hereafter, we shall show that $\{(S_r\mathbf{f}_k)\chi_{B(x_0,R)}\}^N_{k=1}$ is a finite $\frac{\epsilon}{3}$-net of $\{(S_r\mathbf{f})\chi_{B(x_0,R)}\}_{\mathbf{f}\in \mathcal{F}}$ in $L^p(\rho,\mu)$. Note that by Lemma \ref{lem-John},
\begin{align*}
&\|(S_r\mathbf{f})\chi_{B(x_0,R)}-(S_r\mathbf{f}_k)\chi_{B(x_0,R)}\|_{L^{p}(\rho,\mu)}\\
&\quad\leq \bigg(\int_{B(x_0,R)}|W_x(S_r\mathbf{f}(x)-S_r\mathbf{f}_k(x))|^p d\mu(x)\bigg)^{\frac{1}{p}}\\
&\quad\leq \bigg(\int_{B(x_0,R)}\|W_x\|^p_{op}|S_r\mathbf{f}(x)-S_r\mathbf{f}_k(x)|^p d\mu(x)\bigg)^{\frac{1}{p}}\\
&\quad\leq \bigg(\int_{B(x_0,R)}\|W_x\|^p_{op}d\mu(x)\bigg)^{\frac{1}{p}} \sup_{x\in \bar{B}(x_0,R)}|S_r\mathbf{f}(x)-S_r\mathbf{f}_k(x)|< \frac{\epsilon}{3},
\end{align*}
which finishes the proof of Theorem \ref{thm-KRrho}.
\end{proof}

\begin{rk}\label{rk-KRrho}
$(\rm{i})$ By using the same argument in Theorem \ref{thm-KRrho} with some minor changes, one can prove that Theorem \ref{thm-KRrho} also holds for $p=1$ under the additional assumptions that $W_x$ is an invertible matrix weight satisfying \eqref{eq-2.3} for every $x\in X$ and $\|W^{-1}_x\|_{op}\in L^{\infty}_{\rm{loc}}(X,d,\mu).$\par

$(\rm{ii})$ The assumption on $\rho_x$ in Theorem \ref{thm-KRrho} is necessary for our method. Since our method relies on the structure of Banach function spaces (see \cite[Definition 2.1]{GR}), although matrix weighted Lebesgue spaces are not Banach lattices (see \cite{CGO}). Moreover, Tsuji's method is invalid here to relax the range of the exponent $p\in (0,\infty)$.
\end{rk}

As an application, we obtain the following compactness criterion in matrix weighted Lebesgue spaces on $(\mathbb{R}^n,|\cdot|,\mu)$ with a metrically continuous measure $\mu$, by applying Theorem \ref{thm-KRrho} with $\rho_x(\cdot):=|W^{\frac{1}{p}}(x)\cdot|$. This extends the corresponding results of
\cite[Theorem 5]{CC}, \cite[Lemma 4.1]{XYY} and \cite[Proposition 2.9]{COY}.
\begin{cor}\label{cor-mwKR}
Let $1\leq p<\infty$, $(\mathbb{R}^n,|\cdot|,\mu)$ be the Euclidean metric measure space with a metrically continuous measure $\mu$. Assume that $W$ is an invertible matrix weight satisfying
\begin{enumerate}
\item[$(i)$] $\|W^{-1}\|_{op}\in L^{\infty}_{\rm{loc}}(\mathbb{R}^n,|\cdot|,\mu)$ when $p=1$;
\item[$(ii)$] $\|W^{-1}\|^{\frac{p'}{p}}_{op}\in L^1_{\rm{loc}}(\mathbb{R}^n,|\cdot|,\mu)$ when $p\in(1,\infty)$.
\end{enumerate}
Define $$\|\mathbf{f}\|^p_{L^p(W,\mu)}:=\int_{\mathbb{R}^n}\big|W^{\frac{1}{p}}(x)\mathbf{f}(x)\big|^pd\mu(x).$$
A subset $\mathcal{F}$ of $L^p(W,\mu)$ is totally bounded if the following conditions are valid:
\begin{enumerate}
\item[$(a)$] $\mathcal{F}$ is bounded, i.e. $\sup\limits_{\mathbf{f}\in \mathcal{F}}\|\mathbf{f}\|_{L^p(W,\mu)}<\infty;$
\item[$(b)$] $\mathcal{F}$ uniformly vanishes at infinity, that is, $$\lim_{R\to \infty}\sup_{\mathbf{f}\in \mathcal{F}}\|\mathbf{f}\chi_{B^c(0,R)}\|_{L^p(W,\mu)}=0;$$
\item[$(c)$] $\mathcal{F}$ is equicontinuous, that is, $$\lim_{r\to 0}\sup_{\mathbf{f}\in \mathcal{F}}\|S_r\mathbf{f}-\mathbf{f}\|_{L^p(W,\mu)}=0.$$
\end{enumerate}
\end{cor}

Clearly, by Definition \ref{df-mw}, if $W$ is a matrix $A_p$ weight for $p\in(1,\infty)$, then $W$ satisfies the assumption in Corollary \ref{cor-mwKR} on $(\mathbb{R}^n,|\cdot|,m)$.\par

\section{A characterization of compactness on $\mathbb{R}^n$}\label{s5}
In this section, we give a necessary and sufficient condition for total boundedness of subsets in matrix weighted Lebesgue spaces. Before that, we need some lemmas.\par

For any $0<r<\infty$, let $S_r$ be the average operator on $\mathbb{R}^n$ defined by $$S_r\mathbf{f}(x):=\frac{1}{|B(x,r)|}\int_{B(x,r)}\mathbf{f}(y)\ dy,\ \forall\ \mathbf{f}\in L^p(W).$$ Then we have the following useful lemma.
\begin{lem}\label{lem-nKR}
Let $1<p<\infty$. If $W\in A_p$, then
$$\|S_r\mathbf{f}\|_{L^p(W)}\lesssim \|\mathbf{f}\|_{L^p(W)},\ \forall\ \mathbf{f}\in L^p(W),$$
where the implicit constant depends only on $p$ and $d$.
\end{lem}
\begin{proof}
Note that for any $\mathbf{f}\in L^p(W)$ and $0<r<\infty$,
\begin{align}\label{eq-10}
\bigg|\frac{W^{\frac{1}{p}}(x)}{|B(x,r)|}\int_{B(x,r)}\mathbf{f}(y)dy\bigg|&\lesssim \frac{1}{|B(x,r)|}\int_{B(x,r)}\big|W^{\frac{1}{p}}(x)\mathbf{f}(y)\big|dy\notag\\
&\approx \frac{1}{|B(x,r)|}\int_{B(x,r)}\big|W^{\frac{1}{p}}(x)W^{-\frac{1}{p}}(y)W^{\frac{1}{p}}(y)\mathbf{f}(y)\big|dy\notag\\
&\lesssim M_w\big(W^{\frac{1}{p}}\mathbf{f}\big)(x)\tag{5.1}.
\end{align}
Then from \eqref{eq-10} and Lemma \ref{lem-CGm}, it follows that
\begin{align*}
\|S_r\mathbf{f}\|_{L^p(W)}&\lesssim \|M_{\omega}\big(W^{\frac{1}{p}}\mathbf{f}\big)\|_{L^p(\mathbb{R}^n)}\lesssim \|W^{\frac{1}{p}}\mathbf{f}\|_{L^p(\mathbb{R}^n,\mathbb{C}^d)}\approx\|\mathbf{f}\|_{L^p(W)}.
\end{align*}
This completes the proof of Lemma \ref{lem-nKR}.
\end{proof}\par

The following is a vector-valued extension of the Lebesgue differentiation theorem in the setting of matrix weights.
\begin{lem}\label{lem-ldd}
Let $1<p<\infty$. If $W\in A_p$, then for any $\mathbf{f}\in L^p(W)$,
$$\lim_{r\to 0}|S_r\mathbf{f}(x)-\mathbf{f}(x)|=0\quad \text{a.e.}\ x\in \mathbb{R}^n.$$
\end{lem}
\begin{proof}
First, for any $\mathbf{f}:=(f_1,\cdots,f_d)^T\in L^p(W)$, it suffices to show that $f_i\in L^1_{\rm{loc}}(\mathbb{R}^n)$ for each $1\leq i\leq d$. Since $W\in A_p$, then by \eqref{eq-2.2}, $$|\mathbf{f}|^p=|W^{-\frac{1}{p}}W^{\frac{1}{p}}\mathbf{f}|^p\leq \|W^{-\frac{1}{p}}\|^p_{op}|W^{\frac{1}{p}}\mathbf{f}|^p=\|W^{-1}\|_{op}|W^{\frac{1}{p}}\mathbf{f}|^p.$$
It follows that $$|\mathbf{f}|^p\|W^{-1}\|^{-1}_{op}\leq |W^{\frac{1}{p}}\mathbf{f}|^p.$$
From Lemma \ref{lem-ms}, we conclude that $\|W^{-1}\|^{-1}_{op}$ is a scalar $A_p$ weight, and hence $|\mathbf{f}|\in L^1_{\rm{loc}}(\mathbb{R}^n)$, which implies that $f_i\in L^1_{\rm{loc}}(\mathbb{R}^n)$ for each $1\leq i\leq d$.\par

Hence, by the classical Lebesgue differentiation theorem, for each $1\leq i\leq d$, we have $$\lim_{r\to 0}|S_rf_i(x)-f_i(x)|=0\quad \text{a.e.}\ x\in \mathbb{R}^n.$$
Lemma \ref{lem-ldd} then follows from the fact that for any $x\in \mathbb{R}^n$,
\begin{equation*}
|S_r\mathbf{f}(x)-\mathbf{f}(x)|\leq d^{\frac{1}{2}}\max_{1\leq i\leq d}|S_rf_i(x)-f_i(x)|.\qedhere
\end{equation*}
\end{proof}\par

We now present a characterization for total boundedness of subsets in $L^p(W)$ when $W\in A_p$.
\begin{thm}\label{thm-nec}
Let $1<p<\infty$, and $W\in A_p$. A subset $\mathcal{F}$ of $L^p(W)$ is totally bounded if and only if the following conditions hold:
\begin{enumerate}
\item[$(a)$] $\mathcal{F}$ is bounded, i.e. $\sup\limits_{\mathbf{f}\in \mathcal{F}}\|\mathbf{f}\|_{L^p(W)}<\infty;$
\item[$(b)$] $\mathcal{F}$ uniformly vanishes at infinity, that is, $$\lim_{R\to \infty}\sup_{\mathbf{f}\in \mathcal{F}}\|\mathbf{f}\chi_{B^c(0,R)}\|_{L^p(W)}=0;$$
\item[$(c)$] $\mathcal{F}$ is equicontinuous, that is, $$\lim_{r\to 0}\sup_{\mathbf{f}\in \mathcal{F}}\|S_r\mathbf{f}-\mathbf{f}\|_{L^p(W)}=0.$$
\end{enumerate}
\end{thm}
\begin{proof}
The sufficiency is due to Corollary \ref{cor-mwKR}. We now proof the necessity. Assume that $\mathcal{F}\subset L^p(W)$ is totally bounded. Then for any given $\epsilon>0$, there exists $\{\mathbf{f}_k\}^N_{k=1}\subset\mathcal{F}$ such that $\{\mathbf{f}_k\}^N_{k=1}$ is an $\epsilon$-net of $\mathcal{F}$, which implies that for any $\mathbf{f}\in \mathcal{F}$, there exists $\mathbf{f}_k$ such that $\|\mathbf{f}-\mathbf{f}_k\|_{L^p(W)}<\epsilon.$\par

Clearly, $(a)$ is true. As for $(b)$, for each $1\leq k\leq N$, since $\mathbf{f}_k\in L^p(W)$, by the monotone convergence theorem, there exists $R_k>0$ such that $\|\mathbf{f}_k\chi_{B^c(0,R_k)}\|_{L^p(W)}<\epsilon.$ Set $R:=\max\{R_k:1\leq k\leq N\}$. It follows that for given $\mathbf{f}\in \mathcal{F}$, $$\|\mathbf{f}\chi_{B^c(0,R)}\|_{L^p(W)}\leq \|\mathbf{f}-\mathbf{f}_k\|_{L^p(W)}+\|\mathbf{f}_k\chi_{B^c(0,R)}\|_{L^p(W)}<2\epsilon,$$
which implies $(b)$.\par

As for $(c)$, for each $1\leq k\leq N$, by $W\in A_p$, Lemma \ref{lem-CGm} and \eqref{eq-10},
\begin{align*}
\big|W^{\frac{1}{p}}(x)\big(S_r\mathbf{f}_k(x)-\mathbf{f}_k(x)\big)\big|&\leq \big|W^{\frac{1}{p}}(x)S_r\mathbf{f}_k(x)\big|+\big|W^{\frac{1}{p}}(x)\mathbf{f}_k(x)\big|\\
&\leq M_{\omega}\big(W^{\frac{1}{p}}\mathbf{f}_k\big)(x)+\big|W^{\frac{1}{p}}(x)\mathbf{f}_k(x)\big|\in L^p(\mathbb{R}^n).
\end{align*}
Then by Lemma \ref{lem-ldd} and Lebesgue's dominated convergence theorem, there exists $r>0$ such that for any $h\leq r$, $$\max_{1\leq k\leq N}\|S_h\mathbf{f}_k-\mathbf{f}_k\|_{L^p(W)}<\epsilon.$$
From Lemma \ref{lem-nKR}, it follows that
\begin{align*}
\|S_h\mathbf{f}-\mathbf{f}\|_{L^p(W)}&\leq \|S_h\mathbf{f}-S_h\mathbf{f}_k\|_{L^p(W)}+\|S_h\mathbf{f}_k-\mathbf{f}_k\|_{L^p(W)}+\|\mathbf{f}_k-\mathbf{f}\|_{L^p(W)}\\
&\lesssim \|\mathbf{f}_k-\mathbf{f}\|_{L^p(W)}+\|S_h\mathbf{f}_k-\mathbf{f}_k\|_{L^p(W)}\lesssim \epsilon,
\end{align*}
where implicit constants depend only on $p$ and $d$. This implies $(c)$ and completes the proof of Theorem \ref{thm-nec}.
\end{proof}

\section*{Acknowledgement}
Yang is supported by the National Natural Science Foundation of China (Grant Nos. 11971402 and 11871254). Zhuo is supported by the National Natural Science Foundation of China (Grant Nos. 1187110 and 11701174).


\end{document}